\documentclass{article}
\usepackage{amsthm}
\usepackage{amsmath}
\usepackage{amsfonts}
\usepackage[letterpaper,body={14cm,20cm}, mag=1000]{geometry}
\usepackage{amssymb}
\usepackage{secdot}
\usepackage{setspace}
\usepackage{titlesec}
\titleformat{\section}[runin]{\bfseries\filcenter}{\thesection}{1em}{}

\renewcommand{\thesection}{\arabic{section}}
\title{\large \bf The conjugacy class number $k(G)$ - a different perspective}
\author{\small \bf Deepak Gumber\footnote{Research supported by UGC under the Research Award Scheme}  and Hemant Kalra\\
\small \em School of Mathematics and Computer Applications\\
\small \em Thapar University, Patiala - 147 004,
India\\
}

\date{}
\newtheorem{thm}{Theorem}

\newtheorem{pp}[thm]{Proposition}

\begin{document}
\maketitle

In 1903, Landau \cite{lan} proved that for each $k\ge 1$, the equation
$$1=\frac{1}{m_1}+\frac{1}{m_2}+\cdots +\frac{1}{m_k}\;\;\;\;\;(m_1\ge m_2\ge\cdots\ge m_k)$$
has only finitely many solutions over the positive integers $m_i$. Let $G$ be a finite group and let $k(G)$ denote the number of conjugacy classes $C_1,C_2,\ldots,C_{k(G)}$ of $G$. The above equation holds in $G$ if $k=k(G)$ and $m_i=|G|/|C_i|$. It follows that there are only finitely many 
non-isomorphic finite groups with a given number of  conjugacy classes. The problem of the classification of finite groups with a given number of conjugacy classes has a long history and goes back to Burnside. The interested readers can refer to \cite{lop1} and \cite{lop2} and the papers quoted there for a detailed account of this problem and, in particular, for the list of the groups with up to twelve conjugacy classes.

Let $m(G)$ denote the least positive integer $n$ such that the union of any $n$ distinct non-trivial conjugacy classes of $G$ together with the identity of $G$ is a subgroup of $G$. Observe that $m(G)\le k(G)-1$. In \cite{kalgum}, we have classified all finite groups $G$ for $m(G)=1,2$ and 3. It is proved that $m(G)=1$ if and only if $G$ is an elementary abelian 2-group; $m(G)=2$ if and only if $G$ is isomorphic to $C_3$ or to $S_3$; and $m(G)=3$ if and only if $G$ is isomorphic to one of the groups $C_4,C_2\times C_2,D_4$ and $A_5$. It follows that $m(G)=2$ if and only if $k(G)=3$, and $m(G)=3$ if and only if $k(G)=4$. Therefore, for these two particular cases, $m(G)=k(G)-1$. A natural question which arises here is:
$$\mbox{Is it true that}\; m(G)=k(G)-1\;\mbox{for all}\;m(G)\ge 2\,?$$ 
We answer the question in affirmative in Theorem 2. Our notation is mostly standard and that of \cite{gor}. A subgroup $H$ of $G$ is called a complement for a normal subgroup $N$ of $G$ in $G$ if $G=NH$ and $N\cap H =1$. If $N$ has a complement $H$ in $G$, then we say that $G$ splits over $N$ with $H$ as complement. A group $G$ is called a rational group if every element $x$ of $G$ is conjugate to $x^m$, where $m$ is a natural number co-prime to the order $|x|$ of $x$. To prove our theorem, we need the following deep result \cite[Proposition 21]{kle} from group theory:

\begin{pp}
Let $G$ be a finite rational group with an abelian Sylow $2$-subgroup $H$. Then $H$ is elementary abelian, $G$ splits over $G'$ with $H$ as complement, and $G'$ is a $3$-group.
\end{pp}

\begin{thm}
Let $G$ be a finite group. Then $m(G)=k(G)-1$ for all $m(G)\ge 2$.
\end{thm}
\begin{proof}
First suppose that $G$ is abelian and that $m(G)=m<k(G)-1$. Then $G$ has at least $m+1$ distinct non-trivial conjugacy classes $C_1,C_2,\ldots ,C_{m+1}$. Let $H=1\cup_{i=1}^{m}C_i$ and let $K=1\cup_{i=1}^{m-1}C_i\cup C_{m+1}$. Then $H$ and $K$ are subgroups of $G$ and the subgroup $HK$ of $G$ has order 
$$|HK|=\frac{|H||K|}{|H\cap K|}=\frac{(m+1)^2}{m},$$
which is not a positive integer. It thus follows that $m(G)=k(G)-1$. Now suppose that $G$ is non-abelian and $m(G)<k(G)-1$. Then, for any two distinct non-trivial conjugacy classes $C_1$ and $C_2$, there is a subgroup of $G$ containing $C_1$ but not $C_2$. We conclude from this that every element of $G$ is conjugate to any of its non-trivial powers. In particular, all elements of $G$ are of prime order, possibly for various primes, and, since $x$ is conjugate to $x^{-1}$, $|G|$ is even. It follows that $G$ is a rational group with an elementary abelian Sylow 2-subgroup $H$. Therefore, by Proposition 1, $G$ splits over $G'$ with $H$ as complement and $G'$ is a Sylow 3-subgroup. For $1\ne h\in H$, the mapping 
$\varphi_h :G'\rightarrow G'$ defined by $$\varphi_h(x)=h^{-1}xh$$ is an automorphism of order 2. If $\varphi_h(g)=g$ for some $1\ne g\in G'$, then $gh$ is of composite order, which is not so. Thus $\varphi_h$ is a fixed-point-free automorphism of order 2 and therefore, by \cite[Theorem 10.1.4]{gor}, $G'$ is abelian and $\varphi_h(x)=x^{-1}$ for all $x\in G'$. It follows that 
$$|x^G|=\frac{|G|}{|C_G(x)|}=\frac{|G|}{|G'|}=2$$ for all $x\in G'$ and hence $|G|=2.3^b$ for some $b\ge 1$. If $|G'|>3$, then there exist $x,y\in G'$ such that $\langle x\rangle\ne\langle y\rangle$ and therefore, by hypothesis, there is a subgroup of $G$ containing $x^G$ and $y^G$ but not containing $(xy)^G$. This is not possible and hence $G\simeq S_3$. But in $S_3$, $m(G)=k(G)-1=2$, which is a contradiction to our hypothesis. This completes the proof.
\end{proof}

\vspace{1em}

\noindent{\bf Acknowledgment.} Authors are thankful to Prof. Evgeny Khukhro for useful discussions. The present proof for non-abelian case is inspired by his arguments.

\end{document}